\documentclass[11pt,a4paper]{article}

\usepackage{amsmath,amssymb}
\usepackage{graphicx}

\def\@begintheorem#1#2{\list{}{\thm@body}%
  \item[]{\bf #1~#2.}\quad\it\ignorespaces}
\def\@opargbegintheorem#1#2#3{\list{}{\thm@body}%
  \item[]{\bf #1~#2~\ifrembrks #3\global\rembrksfalse\else (#3)\fi.}%
  \quad\it\ignorespaces}
\def\@endtheorem{\endlist}
\newtheorem{theorem}{Theorem}
\newenvironment{proof}
{\begin{trivlist}\item[]{{\it Proof.}}}{\eop\noindent\end{trivlist}}
\newcommand{\eop}{\hfill{$\Box$}}

\begin{document}
\date{}
  \title{Regular finite planar maps with equal edges\\
  {\small-- Memorandum 1982-12 --}\\
  {\small(July 1982; retyped and slightly edited in December 2009 by Sascha Kurz)}}
  \author{{\sc Aart Blokhuis}\thanks{aartb@win.tue.nl}\\ 
  University of Technology\\
  Department of Mathematics and Computing Science\\
  PO Box 513, Eindhoven\\ 
  The Netherlands}
  \maketitle
  
  \noindent
  \rule{\textwidth}{0.3 mm}
  \begin{abstract}
    \noindent
    There doesn't exists a finite planar map with all edges having the same length, and each vertex on exactly $5$~edges.
  \end{abstract}
  \noindent
  \rule{\textwidth}{0.3 mm}

\section{Introduction}

At the $1981$~meeting for Discrete Geometry in Oberwolfach, H.~Harborth posed the following problem: Is it possible to put a finite set of match-sticks in the plane such that in each endpoint a constant number $k$ of matches meet, and no two match-sticks overlap? Also if possible, what is the minimum number of match-sticks in such a configuration. He proceeded to give minimal examples for $k=2$, $k=3$, see Figure~\ref{fig_min_ex_2_3}, and a possibly non-minimal example for $k=4$, see Figure~\ref{fig_min_ex_4}.

\begin{figure}[htp]
\begin{center}
\includegraphics{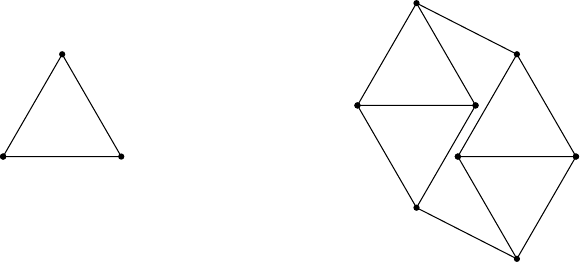}
\caption{Minimal examples of regular match-stick graphs for $k=2$ and $k=3$.}
\label{fig_min_ex_2_3}
\end{center}
\end{figure}

\begin{figure}[htp]
\begin{center}
\includegraphics{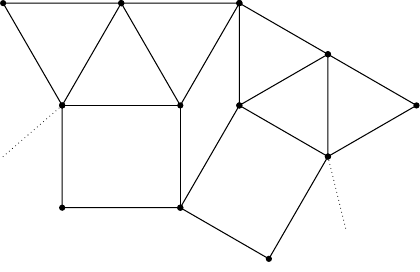}
\caption{The smallest known regular match-stick graph for $k=4$.}
\label{fig_min_ex_4}
\end{center}
\end{figure}

For $k\ge 6$ there exist no finite regular map of valency $k$ by a consequence of Euler's theorem: $|V|-|E|+|F|=2$, where $V$ denotes the set of vertices, $E$ the set of edges, and $F$ the set of faces.

For $k=5$ there do exist finite regular maps, the smallest one is the graph of the icosahedron, see Figure~\ref{fig_icosahedron}, but it is not possible to draw it in such a way that all edges have the same length.

\begin{figure}[htp]
\begin{center}
\includegraphics{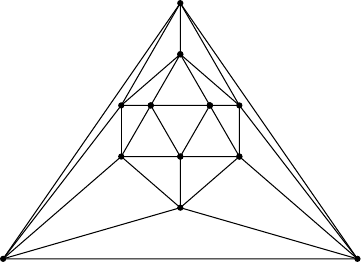}
\caption{The smallest $5$-regular planar map -- the icosahedron.}
\label{fig_icosahedron}
\end{center}
\end{figure}

\noindent
We will show that this is true for all finite planar graphs that are regular of degree~$5$.

\begin{theorem}
  No finite planar map with straight edges of equal length exists that is regular of degree~$5$.
\end{theorem}
\begin{proof}
  Let $V$ denote the set of vertices, $E$ the set of edges, and $F$ the set of faces of a planar map. We then have
  Euler's relation:
  \begin{equation}
    \label{eq_1}
    |V|-|E|+|F|=2.
  \end{equation}
  If, furthermore each point is on $5$~edges then
  \begin{equation}
    \label{eq_2}
    5|V|=2|E|.
  \end{equation}
  Write $F_i$ for the set of faces with $i$~sides, then
  \begin{equation}
    \label{eq_3}
    |F|=\sum\limits_{i=3}^\infty |F_i|=|F_3|+|F_4|+\dots
  \end{equation}
  and
  \begin{equation}
    \label{eq_4}
    2|E|=\sum\limits_{i=3}^\infty i|F_i|=3|F_3|+4|F4|+\dots\,.
  \end{equation}
  We may combine (\ref{eq_2}), (\ref{eq_3}) and (\ref{eq_4}) to get
  \begin{equation}
    \label{eq_5}
    \sum\limits_{i=3}^\infty (10-3i)|F_i|=|F_3|-2|F_4|-5|F_5|-8|F_6|-\dots=20.
  \end{equation}
  For any vertex $v\in V$ we define
  \begin{eqnarray}
    f_i(v) &=& \#\text{ $i$-gonal faces containing $v$},\nonumber\\
    f(v) &=& \sum\limits_{i=3}^\infty \frac{(10-3i)f_i(v)}{i}=\frac{f_3(v)}{3}-\frac{2f_4(v)}{4}
    -\frac{5f_5(v)}{5}-\dots\,.\label{eq_6}
  \end{eqnarray}
  From (\ref{eq_5}), (\ref{eq_6}) and $\sum\limits_{v\in V}\frac{f_i(v)}{i}=|F_i|$ we obtain
  \begin{equation}
    \label{eq_7}
    \sum\limits_{v\in V} f(v)=20.
  \end{equation}
  From now on, we assume that the edges in the map all have the same length. A point is then surrounded by at most $4$~triangles,
  and the only possibilities for a point $v\in V$, making a positive contribution to $\sum\limits_{v\in V} f(v)$ are
  configurations of either four triangles plus a tetragon, or four triangles plus a pentagon, see
  Figure~\ref{fig_positive_contribution}.
\begin{figure}[htp]
\begin{center}
\includegraphics{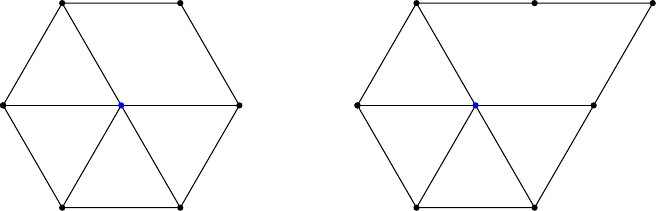}
\caption{Vertex-configurations with a positive contribution.}
\label{fig_positive_contribution}
\end{center}
\end{figure}
  In the first case we have $f(v)=4\cdot\frac{1}{3}-\frac{2}{4}=\frac{5}{6}$ and in the second case we have
  $f(v)=4\cdot\frac{1}{3}-\frac{5}{5}=\frac{1}{3}$. We will show that the positive contribution is killed by
  the surrounding points, yielding $\sum\limits_{v\in V} f(v)\le 0$, which is clearly a contradiction.
  
  First we define a modified map: we add the diagonal in diamonds as in Figure~\ref{fig_diamond}: thus
  producing two equilateral triangles.
\begin{figure}[htp]
\begin{center}
\includegraphics{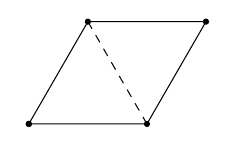}
\caption{Adding a diagonal to a diamond.}
\label{fig_diamond}
\end{center}
\end{figure}
  The effect upon $\sum\limits_{v\in V} f(v)$ is as follows: $f(v_1)$ and $f(v_2)$ are increased by $\frac{1}{3}+\frac{1}{2}$,
  $f(v_3)$ and $f(v_4)$ are increased by $2\cdot\frac{1}{3}+\frac{1}{2}$; therefore each added diagonal produces an increment
  of~$4$:
  \begin{equation}
    \label{eq_8}
    \sum\limits_{v\in V} f(v)=20+4\times(\#\text{ added diagonals}).
  \end{equation}
  After the addition of extra diagonals some points now may have valency~$6$ or $7$.
  (At most two of the five faces at a point can be diamonds.) Let us denote the number of points of valency~$6$ by
  $v_6$ and the number of points of valency~$7$ by $v_7$. With this we have the relation
  \begin{equation}
    \label{eq_9}
    2\times(\#\text{ added diagonals})=v_6+2v_7.
  \end{equation}
  Together with~(\ref{eq_8}) this gives:
  \begin{equation}
    \label{eq_10}
    \sum\limits_{v\in V} f(v)\,-2v_6-4v_7=20.
  \end{equation}
  The contribution of points with valency~$6$ or~$7$ to the left hand side of this relation is non-positive, which shows that
  we may limit our considerations to points that are part of a pentagon, since all other points do not make a positive
  contribution. As defined before $F_5$ denotes the set of pentagonal faces. Let us denote by $V(F_5)$ the set of points
  contained in a pentagonal face.
  
  Let $\widetilde{f}(v)=f(v)-2(d(v)-5)$, i.~e.{} $\widetilde{f}(v)=f(v)$ for vertices of valency~$5$, $\widetilde{f}(v)=f(v)-2$
  for vertices of valency~$6$, and $\widetilde{f}(v)=f(v)-4$ for vertices of valency~$7$, where $d(v)$ is the degree of vertex~$v$.
  With this we can rewrite relation~(\ref{eq_10}) as
  $$
    \sum\limits_{v\in V} \widetilde{f}(v)=20
  $$
  or, separating pentagonal points and non-pentagonal points:
  \begin{equation}
    \label{eq_11}
    \sum\limits_{v\in V\backslash V(F_5)} \widetilde{f}(v)\,+\,\sum\limits_{v\in V(F_5)} \widetilde{f}(v)=20.
  \end{equation}
  Since $\widetilde{f}(v)\le 0$ for all $v\in V\backslash V(F_5)$ we will now investigate
  $$
    \sum\limits_{v\in V(F_5)} \widetilde{f}(v)=\sum\limits_{P\in F_5}\sum\limits_{v\in P} \frac{\widetilde{f}(v)}{f_5(v)}.
  $$
  We will finish the proof by showing that
  $$
    \sum\limits_{v\in P} \frac{\widetilde{f}(v)}{f_5(v)}\le 0
  $$
  for all pentagons $P\in F_5$.
  
  Now let us classify the situations where $\frac{\widetilde{f}(v)}{f_5(v)}>-\frac{1}{2}$. If $f_5(v)\ge 2$ then we have
  $$\frac{\widetilde{f}(v)}{f_5(v)}\le \frac{1}{f_5(v)}\cdot\left(f_5(v)\cdot 
  -\frac{5}{5}+\Big(d(v)-f_5(v)\Big)\cdot\frac{1}{3}-2(d(v)-5)\right)\le-\frac{1}{2}.$$
  Thus it remains to consider the case $f_5(v)=1$, where we have
  \begin{eqnarray*}
    \frac{\widetilde{f}(v)}{f_5(v)}&\le& -\frac{5}{5}+f_3(v)\cdot\frac{1}{3}+\Big(d(v)-f_3(v)-1)\Big)\cdot -\frac{2}{4}-2(d(v)-5)\\
    &=& \frac{5f_3(v)}{6}+\frac{19}{2}-\frac{5d(v)}{2}.
  \end{eqnarray*}
  Since $f_3(v)\le d(v)-1$ vertices with $d(v)\ge 6$ fulfill $\frac{\widetilde{f}(v)}{f_5(v)}\le -\frac{4}{3}$. For $d(v)=5$,
  $f_3(v)\le 3$ we have $\frac{\widetilde{f}(v)}{f_5(v)}\le -\frac{1}{2}$ and for $d(v)=5$, $f_3(v)=4$ we have
  $\frac{\widetilde{f}(v)}{f_5(v)}=\frac{1}{3}$.
  
  Thus the only configuration with positive $\frac{\widetilde{f}(v)}{f_5(v)}$ consists of four triangles
  and a pentagon. Now suppose that at least three vertices $v$ of $P$ are of this type. In this case two of them have to
  be neighbored, let us call them $v_1$ and $v_2$, and the pentagon~$P$ has to be in the shape of a chain of three
  equilateral triangles, see Figure~\ref{fig_pentagon}. 
\begin{figure}[htp]
\begin{center}
\includegraphics{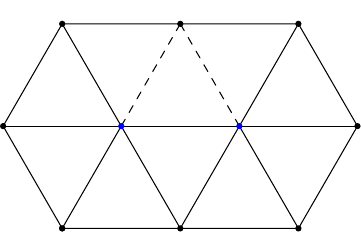}
\caption{A pentagon with at least three vertices fulfilling $\frac{\widetilde{f}(v)}{f_5(v)}>-\frac{1}{2}$.}
\label{fig_pentagon}
\end{center}
\end{figure}
  Additionally only that vertex of $P$ which is not neighbored to either $v_1$ or
  to $v_2$ can have $\frac{\widetilde{f}(v)}{f_5(v)}>-\frac{1}{2}$. So we finally conclude
  $$
    \sum\limits_{v\in P} \frac{\widetilde{f}(v)}{f_5(v)}\le \max\limits_{x\in\{0,1,2,3\}}
    \left\{ x\cdot\frac{1}{3}+(5-x)\cdot-\frac{1}{2}\right\}=0.
  $$
\end{proof}

\medskip

\noindent
We would like to remark that the proof can be slightly modified to prove that there doesn't exist a  finite planar map with minimum degree~$5$ and all edges having the same length. Instead of Equation~(\ref{eq_5}) we then obtain
\begin{eqnarray}
  && \sum\limits_{i=3}^\infty (10-3i)|F_i|=20+\sum\limits_{i=6}^\infty 2(i-5)v_i\\
  &\Longleftrightarrow&  |F_3|-2|F_4|-5|F_5|-8|F_6|-\dots=20+2v_6+4v_7+6v_8+\dots.\nonumber
\end{eqnarray}
With this we define
\begin{equation}
  \hat{f}(v) =\sum\limits_{i=3}^\infty \frac{(10-3i)f_i(v)}{i}-2\Big(d(v)-5\Big)=10-2d(v)+\frac{f_3(v)}{3}-\frac{2f_4(v)}{4}
    -\frac{5f_5(v)}{5}-\dots
  \end{equation}
for each vertex $v\in V$ and obtain $\sum\limits_{v\in V} \hat{f}(v)=20$. Since we may add edges without violating either the degree condition or the condition on the edge lengths,  we can assume that the planar map does not contain any diamonds. As in the proof above $\hat{f}(v)>0$ is only possible if $f_5(v)>0$. The proof is finished by showing that
\begin{equation}
  \sum\limits_{v\in P} \frac{\hat{f}(v)}{f_5(v)}\le 0
\end{equation}
holds for all pentagons $P\in F_5$.

\section*{Epilog}
Another proof that there does not exist a $5$-regular matchstick graph is published in \cite{short}.

\providecommand{\href}[2]{#2}

\end{document}